\documentclass[smallcondensed]{svjour3}
\usepackage[table]{xcolor}
\usepackage{amsfonts}
\usepackage{latexsym}
\usepackage{amsmath}
\usepackage{amssymb}
\usepackage{amscd}
\usepackage{epsfig}
\usepackage{graphicx, subfigure}
\usepackage{enumitem}
\usepackage{pgfplots}
\pgfplotsset{width=7cm}
\renewenvironment{proof}{\par {\sc {\bf Proof.}\hskip 5pt}}{\hfill \qed \par}

\begin{document}

\title{An improved binary programming formulation for the secure domination problem}
\author{Ryan Burdett, Michael Haythorpe}

\maketitle

\begin{abstract}The secure domination problem, a variation of the domination problem with some important real-world applications, is considered. Very few algorithmic attempts to solve this problem have been presented in literature, and the most successful to date is a binary programming formulation which is solved using CPLEX. A new binary programming formulation is proposed here which requires fewer constraints and fewer binary variables than the existing formulation. It is implemented in CPLEX, and tested on certain families of graphs that have previously been considered in the context of secure domination. It is shown that the runtime required for the new formulation to solve the instances is significantly less than that of the existing formulation. An extension of our formulation that solves the related, but further constrained, secure connected domination problem is also given; to the best of the authors' knowledge, this is the first such formulation in literature.

%

\end{abstract}

\section{Introduction}

In this manuscript we consider the {\em secure domination problem}, which is a recently introduced special case of the more famous {\em domination problem}. Consider a graph $G = \langle V,E \rangle$, with vertex set $V$ and edge set $E$. We shall use the convention that $n = |V|$ is the order of the graph, and $m = |E|$ is the size of the graph. We will use $N(i)$ to denote the subset of vertices $\{j \; | \; j \in V, (i,j) \in E\}$, and $N[i] = \{N(i),i\}$. We refer to $N(i)$ as the {\em open neighbourhood of $i$} and $N[i]$ as the {\em closed neighbourhood of $i$}. Then a {\em dominating set} of $G$ is a subset of the vertices $D$ that satisfies the following: for every vertex $v \in V$, at least one vertex in the closed neighbourhood of $v$ is contained in $D$. That is, every vertex in the graph is either in $D$, or is adjacent to a vertex in $D$; we refer to vertices for which this is true as being {\em covered}. Then the {\em domination problem} for a given graph $G$ is to find the dominating set $D^*$ of smallest cardinality for that graph, where the cardinality of $D^*$ is called the {\em domination number} $\gamma(G)$ of the graph.


There are many real-world interpretations of the domination problem, and we outline a common one here as follows. Suppose that the graph $G$ corresponds to a facility, and each vertex $V$ corresponds to a location in the facility. If you can observe location $i$ from location $j$, then $(i,j)$ is an edge in $G$. Then, if a dominating set $D$ is found, and a guard is placed at each of the locations in $D$, every location in the facility is being observed by at least one guard. The domination problem attempts to do so with the fewest guards.

The domination problem has been a topic of research for over a century, with roots in chess and recreational mathematics. An early problem involved finding the minimum number of queens required to cover an $n\times n$ chess board, in the sense that at least one queen could move to any given square on the board. Research has been carried out more extensively in recent years due to many applications in the realm of mathematics, computer science, social sciences and so on. A wealth of information on graph domination and the domination problem can be found in books by Haynes, Hedetniemi and Slater \cite{haynes}. In terms of calculating the domination number for arbitrary graphs, the current fastest algorithm developed by van Rooij and Bodlaender \cite{vanrooij} uses the {\em branch-and-reduce} paradigm, paired with a {\em measure and conquer} analysis method.

In 2005, Cockayne \cite{cockayne} introduced the concept of a {\em secure dominating set}. The interpretation is as follows. Suppose we have a facility $G$ containing locations $V$, and guards stationed at various locations corresponding to a dominating set $D$, as described previously. Now, suppose that there is an incursion at one of the vertices in $V$. Either a guard is already located at that vertex, or a guard from an adjacent vertex has moved to handle the incursion. However, if a guard moves, it is possible that this may leave another vertex unobserved, which may be undesirable. With this interpretation in mind, a dominating set $D$ is said to be {\em secure} if, for every vertex $v$, there exists a vertex $w \in N[v] \cap D$ such that $(D \setminus w) \cup v$ is also dominating set. In this case, we say that $w$ {\em defends} $v$. Then, the {\em secure domination problem} for a given graph $G$ is to find the secure dominating set $S^*$ of smallest cardinality for that graph, and the {\em secure domination number} $\gamma_s(G) = |S^*|$.

Note that, for a given vertex $v$, there may be multiple vertices capable of defending it. In the case of an incursion at $v$, only one of these, say $w$, is required to do so. In such a case, we will say that $w$ {\em moves to} $v$. Then, in the resulting configuration, a guard is now present at $v$, but there is no longer a guard present at $w$. Informally then, the requirement of a secure dominating set is that this new configuration is itself a dominating set.

In the years since Cockayne introduced the notion of secure dominating sets, there has been a moderate amount of research into the problem. It is obvious that the problem is NP-complete in general, and it has also been shown to be NP-complete for certain special cases \cite{merouane,pradhan,wang}. Bounds for the secure domination number have been computed for some specific graph families \cite{araki,cockayne,winter}. Finally, there has also been some effort put into developing algorithms for computing the secure domination number. For the specific cases of trees and block graphs, linear-time algorithms have been developed \cite{burgerlinear,pradhan}. For general graphs, very few attempts at developing algorithms have been made to date. In \cite{burgerbnb}, implementations of both a branch-and-bound algorithm, and a branch-and-reduce algorithm were presented. The same authors also formulated the problem as a binary program \cite{burgerformulation}, which they used CPLEX to solve. Their experiments indicated that the binary programming formulation was significantly faster than either of their other algorithmic approaches. Hence, we focus only on their binary programming formulation here.

The formulation in \cite{burgerformulation} includes two sets of binary variables; first, $x_i = 1$ if vertex $i$ is to be included in the optimal secure dominating set, and $x_i = 0$ otherwise. Then, they define {\em swap} variables $z_{kl}$ such that $z_{kl} = 1$ implies that if the guard at vertex $l$ moves to vertex $k$, the resulting configuration is still dominating. By definition, $z_{kl} = 0$ if $l \not\in N(k)$, and so we only consider $z_{kl}$ to be defined if $l \in N(k)$. Hence, there are $n + m$ binary variables. Then their formulation can be paraphrased as follows on the next page.

Constraints (\ref{eq1})--(\ref{eq6}) have fairly natural interpretations. Constraints (\ref{eq1}) ensure that each vertex is either in the secure dominating set, or is adjacent to a vertex that is. Constraints (\ref{eq2}) ensure that each vertex is either guarded, or a neighbouring guard is in the swap set to defend it. Constraints (\ref{eq3}) ensure that $z_{kl}$ is only permitted to be equal to 1 if there is a guard at vertex $l$, and no guard at vertex $k$. Constraints (\ref{eq4}) ensure that, in the event that the guard at vertex $l$ moves to vertex $k$, every other vertex is still covered. Finally, constraints (\ref{eq5})--(\ref{eq6}) ensure that the variables are binary. It can be easily checked that there are $O(nm)$ constraints.

$$\mbox{minimise} \sum_{i=1}^n x_i$$

subject to the constraints
\begin{eqnarray}
\sum_{j \in N[i]} x_j & \geq & 1, \qquad\qquad \forall i = 1, \hdots, n,\label{eq1}\\
x_k + \sum_{l \in N(k)} z_{kl} & \geq & 1, \qquad\qquad \forall k = 1, \hdots, n,\label{eq2}\\
x_k - x_l + 2z_{kl} & \leq & 1, \qquad\qquad \forall k = 1, \hdots, n; \quad j \in N(k),\label{eq3}\\
z_{kl} - \sum_{j \in N[i] \setminus l} x_j & \leq & 0, \qquad\qquad \forall i,k = 1, \hdots, n; \quad i \not\in N[k]; \quad l \in N(k),\label{eq4}\\
x_k & \in & \{0,1\} \qquad \forall k = 1, \hdots, n,\label{eq5}\\
z_{kl} & \in & \{0,1\} \qquad \forall k = 1, \hdots, n; \quad l \in N(k).\label{eq6}\end{eqnarray}

In what follows, we will introduce an alternative binary programming formulation for the secure domination problem. Unlike the formulation discussed previously, it will be shown that in our formulation, only $n$ binary variables and $O(m + n^2)$ constraints are required. We will demonstrate that this leads to significant improvement in algorithmic performance. Finally, we will also show that our formulation can be augmented to provide a formulation of the {\em secure connected domination} problem. To the best of our knowledge, this is the first such formulation to be described in literature.

\section{Improved Formulation}

In the new formulation we propose here, we again use the $x_i$ variables as described above. We will also use new, continuous, variables $y_{ij}$. In our formulation, we will aim to have at least one {\em designated guard} for each vertex of the graph for which a guard is not already present. Then, the interpretation is that $y_{ij} > 0$ if the guard at vertex $i$ is a designated guard for vertex $j$, and $y_{ij} = 0$ otherwise. By definition, $y_{ij} = 0$ if $(i,j) \not\in E$, and so there are as many $y_{ij}$ variables as there are edges.

The new formulation is as follows.

%

$$\mbox{minimise} \sum_{i=1}^n x_i$$

subject to the constraints

\begin{eqnarray}
\sum_{j \in N[i]} x_j & \geq & 1, \qquad\qquad \forall i = 1, \hdots, n,\label{neweq1}\\
y_{ij} - x_i & \leq & 0, \qquad\qquad \forall i = 1, \hdots, n; \quad \forall j \in N(i),\label{neweq2}\\
\sum_{k \in N[i]} x_k - \sum_{\substack{k \in N(i)\\k \in N(j)}} y_{kj} & \geq & 1, \qquad\qquad \forall j = 1, \hdots, n; \quad \forall i\mbox{ such that } dist(i,j) = 2,\label{neweq3}\\
x_j + \sum_{i \in N(j)} y_{ij} & = & 1, \qquad\qquad \forall j = 1, \hdots, n,\label{neweq4}\\
y_{ij} & \geq & 0, \qquad\qquad \forall i = 1, \hdots, n; \quad \forall j \in N(i).\label{neweq5}\\
x_j & \in & \{0,1\}, \qquad \forall i = 1, \hdots, n,\label{neweq6}\end{eqnarray}

\begin{theorem}The set of solutions to (\ref{neweq1})--(\ref{neweq6}) correspond precisely to the secure dominating sets for a given graph.\end{theorem}

\begin{proof}
In order to argue that the above formulation has solutions corresponding precisely to the secure dominating sets, we first outline the three requirements of a secure dominating set.

\begin{enumerate}\item The set itself must be a dominating set.
\item Each vertex $j$ must either be guarded, or have at least one designated guard chosen from its adjacent vertices.
\item If a designated guard at vertex $k$ moves to vertex $j$, then any vertices in $N(k)$ must still be covered.\end{enumerate}

It is clear that equations (\ref{neweq1}) and (\ref{neweq6}) satisfy requirement 1. It can also been seen that equations (\ref{neweq2}), (\ref{neweq4}) and (\ref{neweq5}) satisfy requirement 2, by the following argument. First, constraints (\ref{neweq2}) and (\ref{neweq5}) ensure that if $x_i = 0$, then $y_{ij} = 0$. Hence, there must be a guard at vertex $i$ in order to choose it as a designated guard for another vertex. Then, for each vertex $j$, constraint (\ref{neweq4}) ensures that either $x_j = 1$ or $\sum_{i \in N(i)} y_{ij} = 1$; this corresponds precisely to requirement 2.

Finally, we focus on requirement 3. We will show that constraint (\ref{neweq3}) ensures that requirement 3 is satisfied in all scenarios where a designated guard moves, and is redundant in all other scenarios. Suppose that there is an incursion at vertex $j$. Since requirement 2 is satisfied, there are two possibilities; either there is a guard present at vertex $j$, or vertex $j$ has at least one designated guard at some other vertex $k$. Consider first the case where there is a guard present at vertex $j$; clearly, no designated guard needs to move in this case. Then, by constraint (\ref{neweq4}) it is clear that $y_{kj} = 0$ for all $k$, and hence constraint (\ref{neweq3}) reduces to constraint (\ref{neweq1}).

Next, consider the case where there is no guard present at vertex $j$, and a designated guard at vertex $k$ moves to $j$. Then, consider a vertex $i \in N(k)$. This vertex still needs to be covered once the guard moves away from vertex $k$. If $i \in N(j)$ then it is still covered after the designated guard moves. Hence, we only need to focus on the case when $i \not\in N(j)$. This corresponds precisely to the situation when $dist(i,j) = 2$. Then it is clear that $y_{kj} > 0$. Therefore, constraint (\ref{neweq3}) reduces to $\sum_{k \in N[i]} x_k > 1$. By definition, $x_k = 1$, and so there must be at least one other vertex in $N[i]$ in the dominating set. Hence, vertex $i$ is still covered after the guard at vertex $k$ moves, satisfying requirement 3.\end{proof}

We now focus on the number of constraints in our new formulation. Excluding constraints (\ref{neweq3}) it is easy to check that there are $O(m)$ constraints. In general, the number of constraints in (\ref{neweq3}) depends on the graph, however in the worst case it could be $O(n^2)$. Hence the number of constraints is $O(m + n^2)$. However, we note that if the graph is sparse, it is likely that there are only $O(m) = O(n)$ constraints in (\ref{neweq3}), and hence there are $O(n)$ constraints overall, compared to $O(n^2)$ constraints in (\ref{eq1})--(\ref{eq6}). Also, if the graph is dense, then $m = O(n^2)$, and so there are $O(n^2)$ constraints overall, compared to $O(n^3)$ constraints in (\ref{eq1})--(\ref{eq6}). In either case, in general we expect a factor of $n$ fewer constraints in our formulation.

\section{Experimental Results}

We now present some experimental results demonstrating the improvements obtained by our formulation. Each example was run using CPLEX v12.9.0 via the Concert environment to MATLAB R2018b (9.5.0.944444). The experiments were conducted on an Intel(R) Core(TM) i7-4790 CPU with a 4 core, 3.60GHz processor and 8GB of RAM, running Windows 10 Enterprise version 1803. In each case, we set a maximum time limit of 10,000 seconds. It is worth noting that although there is an overhead involved with using the Concert environment in MATLAB, the times we report here are solely CPLEX runtimes.

We will consider the following six sets of instances.

\begin{enumerate} \item {\bf Square grid graphs} $\mathcal{G}_{k,k}$, which correspond to the Cartesian product of two equal-length paths. These were used in \cite{burgerformulation} when first benchmarking their formulation.
\item {\bf Hexagonal grid graphs} $\mathcal{H}_{k,k}$, which were also used in \cite{burgerformulation} when first benchmarking their formulation.
\item {\bf Queen graphs} $Q_k$, a generalisation of the Queens domination problem first formally proposed by de Jaenisch \cite{dejaenisch} to $n \times n$ sized boards.
\item {\bf Torus grid graphs} $T_{n,n}$, which correspond to the Cartesian product of two equal-length cycles.
\item {\bf Generalized Petersen graphs} $P(k,1)$, a family of 3-regular graphs first named by Watkins \cite{watkins}.
\item {\bf Generalized Petersen graphs} $P(k,2)$, a family of 3-regular graphs first named by Watkins \cite{watkins}.
\end{enumerate}

These sets of instances were chosen because in each case, there are known upper bounds for the secure domination number. Specifically, upper bounds for sets 1 and 4 are found in \cite{cockayne}, for set 2 are found in \cite{burgerformulation}, and for the remaining sets are found in \cite{winter}. In the case of the Hexagonal grid graphs and Queen graphs, the upper bounds are given only for $k \leq 10$.

\begin{figure}[p!]
{
                \centering
\begin{tikzpicture} 
\begin{axis}[
title={$k\times k$ Queens Graphs},
xlabel={k},
ylabel={time (seconds)},
xmin=2, xmax=12,
ymin=0, ymax=20000,
xtick={2,3,4,5,6,7,8,9,10,11,12},
ytick={0,0.01,0.1,1,10,100,1000,10000},
ymode=log,
legend pos=north west,
grid style=dashed,
]

\addplot[
dashed
]
coordinates {
                (2,0.015)
                (3,0.11)
                (4,0.094)
                (5,2.062)
                (6,762.546)
                (7,10000)
                (12,10000)

};

\addplot[
]
coordinates {
                (2,0.015)
                (3,0)
                (4,0.016)
                (5,0.266)
                (6,1.797)
                (7,8.875)
                (8,109.765)
                (9,3357.234)
                (10,6417.468)
                (11,10000)
                (12,10000)
                };

\end{axis}

\end{tikzpicture}
\hskip 5pt
\begin{tikzpicture} 
\begin{axis}[
title={$k \times k$ Grid Graphs},
xlabel={k},
ylabel={time (seconds)},
xmin=2, xmax=13,
ymin=0, ymax=20000,
xtick={2,3,4,5,6,7,8,9,10,11,12,13},
ytick={0,0.01,0.1,1,10,100,1000,10000},
ymode=log,
legend pos=north west,
grid style=dashed,
]

\addplot[
dashed
]
coordinates {
                (1,0.016)
                (2,0)
                (3,0.016)
                (4,0.093)
                (5,0.188)
                (6,0.89)
                (7,5.671)
                (8,10000)
                (13,10000)

};

\addplot[
]
coordinates {
                (1,0)
                (2,0.016)
                (3,0.016)
                (4,0.015)
                (5,0.032)
                (6,0.031)
                (7,0.406)
                (8,0.782)
                (9,1.609)
                (10,109.969)
                (11,4144.422)
                (12,10000)
                (13,10000)

};

\end{axis}

\end{tikzpicture}

\vspace*{1cm}

\begin{tikzpicture} 
\begin{axis}[
title={$k\times k$ Hexagonal Grid},
xlabel={k},
ylabel={time (seconds)},
xmin=2, xmax=14,
ymin=0, ymax=20000,
xtick={2,3,4,5,6,7,8,9,10,11,12,13,14},
ytick={0,0.01,0.1,1,10,100,1000,10000},
ymode=log,
legend pos=north west,
grid style=dashed,
]

\addplot[
dashed
]
coordinates {
                (2,0.016)
                (3,0.016)
                (4,0.078)
                (5,0.078)
                (6,0.875)
                (7,25.094)
                (8,2092.422)
                (9,10000)
                (14,10000)

};

\addplot[
]
coordinates {
                (2,0.016)
                (3,0.016)
                (4,0.016)
                (5,0.016)
                (6,0.047)
                (7,0.079)
                (8,0.203)
                (9,0.921)
                (10,26.109)
                (11,275.187)
                (12,3587.438)
                (13,10000)
                (14,10000)

};

\end{axis}

\end{tikzpicture}
\hskip 5pt
\begin{tikzpicture} 
\begin{axis}[
title={$k\times k$ Torus Graphs},
xlabel={k},
ylabel={time (seconds)},
xmin=2, xmax=12,
ymin=0, ymax=20000,
xtick={2,3,4,5,6,7,8,9,10,11,12},
ytick={0,0.01,0.1,1,10,100,1000,10000},
ymode=log,
legend pos=north west,
grid style=dashed,
]

\addplot[
dashed
]
coordinates {
                (2,0.094)
                (3,0.094)
                (4,0.14)
                (5,0.437)
                (6,1.469)
                (7,32.125)
                (8,10000)
                (12,10000)

};

\addplot[
]
coordinates {
                (2,0.015)
                (3,0.015)
                (4,0.016)
                (5,0.063)
                (6,0.109)
                (7,0.859)
                (8,10.828)
                (9,25.375)
                (10,10000)
                (12,10000)

};

\end{axis}

\end{tikzpicture}

\vspace*{1cm}

\begin{tikzpicture} 
\begin{axis}[
title={GP Graphs (k,1)},
xlabel={k},
ylabel={time (seconds)},
xmin=10, xmax=90,
ymin=0, ymax=20000,
xtick={10,20,30,40,50,60,70,80,90},
ytick={0,0.01,0.1,1,10,100,1000,10000},
ymode=log,
legend pos=north west,
grid style=dashed,
]

\addplot[
dashed
]
coordinates {
(4,0.016)
(5,0.015)
(6,0.015)
(7,0.031)
(8,0.032)
(9,0.063)
(10,0.062)
(11,0.141)
(12,0.203)
(13,0.204)
(14,0.172)
(15,0.516)
(16,0.328)
(17,0.453)
(18,0.469)
(19,0.906)
(20,0.922)
(21,0.906)
(22,1.062)
(23,1.563)
(24,1.531)
(25,1.5)
(26,1.875)
(27,2.75)
(28,6.016)
(29,2.235)
(30,3.297)
(31,7.922)
(32,6.328)
(33,6.375)
(34,8.344)
(35,19.75)
(36,35.985)
(37,11.219)
(38,24.734)
(39,36.625)
(40,17.875)
(41,20.984)
(42,39.312)
(43,59.859)
(44,10000)
(90,10000)
};

\addplot[
]
coordinates {
	(4,0.359)
	(5,0.125)
	(6,0.359)
	(7,0.141)
	(8,0.031)
	(9,0.062)
	(10,0.063)
	(11,0.188)
	(12,0.062)
	(13,0.032)
	(14,0.062)
	(15,0.032)
	(16,0.047)
	(17,0.062)
	(18,0.062)
	(19,0.156)
	(20,0.172)
	(21,0.094)
	(22,0.093)
	(23,0.093)
	(24,0.078)
	(25,0.094)
	(26,0.188)
	(27,0.25)
	(28,0.265)
	(29,0.25)
	(30,0.328)
	(31,0.328)
	(32,0.203)
	(33,0.328)
	(34,0.344)
	(35,0.5)
	(36,0.766)
	(37,0.391)
	(38,0.594)
	(39,0.672)
	(40,0.625)
	(41,0.61)
	(42,0.859)
	(43,1.437)
	(44,5.859)
	(45,0.859)
	(46,1.594)
	(47,7.25)
	(48,0.906)
	(49,1.531)
	(50,7.235)
	(51,9.734)
	(52,24.657)
	(53,9.969)
	(54,12.375)
	(55,24.656)
	(56,5.782)
	(57,13.703)
	(58,26.14)
	(59,34.5)
	(60,148.063)
	(61,20.016)
	(62,82.906)
	(63,143.219)
	(64,18.063)
	(65,62.594)
	(66,168.86)
	(67,245.109)
	(68,1158.609)
	(69,102.906)
	(70,490.312)
	(71,903.875)
	(72,100.359)
	(73,245.188)
	(74,908.406)
	(75,3355.407)
	(76,9460.312)
	(77,564)
	(78,3093.781)
	(79,7769.844)
	(80,1056.766)
	(81,2740.984)
	(82,9210.75)
	(83,10000)
    (84,10000)
    (85,4656.92)
    (86,10000)
	(90,10000)
	
};

\end{axis}

\end{tikzpicture}
\hskip 5pt
\begin{tikzpicture} 
\begin{axis}[
title={GP Graphs (k,2)},
xlabel={k},
ylabel={time (seconds)},
xmin=10, xmax=70,
ymin=0, ymax=20000,
xtick={10,20,30,40,50,60,70},
ytick={0,0.01,0.1,1,10,100,1000,10000},
ymode=log,
legend pos=north west,
grid style=dashed,
]

\addplot[
dashed
]
coordinates {
	(5,0.172)
	(6,0.235)
	(7,0.094)
	(8,0.063)
	(9,0.062)
	(10,0.094)
	(11,0.172)
	(12,0.157)
	(13,0.235)
	(14,0.188)
	(15,0.344)
	(16,0.422)
	(17,0.563)
	(18,1.156)
	(19,0.625)
	(20,0.782)
	(21,0.765)
	(22,1.125)
	(23,1.125)
	(24,1.609)
	(25,1.938)
	(26,1.39)
	(27,2.547)
	(28,2.671)
	(29,6.75)
	(30,6)
	(31,6.922)
	(32,9.359)
	(33,8.984)
	(34,12.766)
	(35,14.704)
	(36,8.031)
	(37,11.828)
	(38,21.75)
	(39,50.688)
	(40,19.375)
	(41,45.188)
	(42,22.781)
	(43,26.359)
	(44,10000)
	(70,10000)

};

\addplot[
]
coordinates {
	(5,0.031)
	(6,0.016)
	(7,0.015)
	(8,0.031)
	(9,0.016)
	(10,0.016)
	(11,0.031)
	(12,0.062)
	(13,0.062)
	(14,0.047)
	(15,0.047)
	(16,0.063)
	(17,0.063)
	(18,0.203)
	(19,0.094)
	(20,0.094)
	(21,0.219)
	(22,0.125)
	(23,0.172)
	(24,0.25)
	(25,0.531)
	(26,0.313)
	(27,0.375)
	(28,0.281)
	(29,0.391)
	(30,0.453)
	(31,1.078)
	(32,1.406)
	(33,1.203)
	(34,1.578)
	(35,0.813)
	(36,0.625)
	(37,2.656)
	(38,10.015)
	(39,20.89)
	(40,6.078)
	(41,7.922)
	(42,13.938)
	(43,10.875)
	(44,20.829)
	(45,57.969)
	(46,65.328)
	(47,46.109)
	(48,54.015)
	(49,38.89)
	(50,34.453)
	(51,151.156)
	(52,423.375)
	(53,1336.079)
	(54,259.859)
	(55,600)
	(56,883.875)
	(57,470.094)
	(58,1426.781)
	(59,1829.922)
	(60,10000)
	(70,10000)
	
};

\end{axis}

\end{tikzpicture}

}

\caption{Runtimes for each of the six families of instances considered. The dotted line corresponds to the runtime for the formulation by Burger et al. \cite{burgerformulation}, while the solid line corresponds to the formulation in this paper. If the runtime for a given instance is longer than 10,000 seconds, we consider that instance to have timed out, and set the final runtime to 10,000 seconds in the figure. Note that the instances in the first four figures have order $k^2$, while the instances in the final two figures have order $2k$.\label{fig1}}
\end{figure}
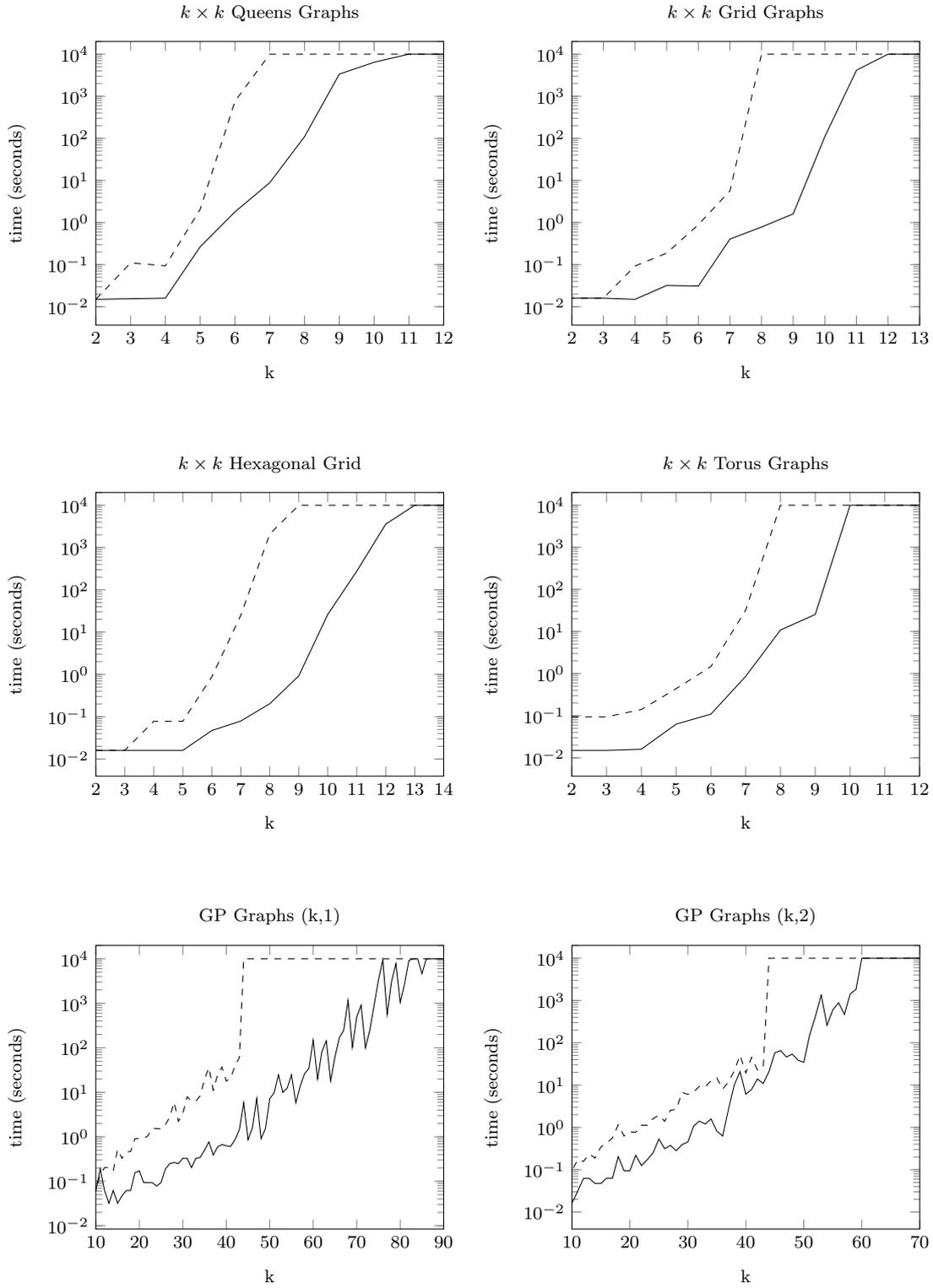

Figure \ref{fig1} displays the runtimes for each of the six sets of instances for both our formulation, and the formulation contained in \cite{burgerformulation}. If the instance takes longer than the allowed time of 10,000 seconds, we consider the instance to have timed out, and so we set the final time to 10,000 seconds in Figure \ref{fig1}. As can be seen, in all tested cases, our formulation is able to discover a minimum secure dominating set in less time than that of the formulation by Burger et al. \cite{burgerformulation}, and hence is able to solve larger instances for each of the sets in the allowed time.

Since there are very few instances where the secure domination number is known, we include here the exact values for the instances we were able to solve, so that they can be used by future researchers when evaluating heuristics or other formulations. These are displayed in Table \ref{tab1}. We make one small note here; in the paper by Burger et al. \cite{burgerformulation} they list the secure domination number for the $9 \times 9$ grid graph as 29. However, the correct answer appears to be 28. We display such a configuration on the $9 \times 9$ grid graph in Figure \ref{fig2}.

\definecolor{lightergray}{gray}{0.9}
\definecolor{midgray}{gray}{0.6}
\begin{table}[h!]\begin{center}
\begin{tabular}{cccccc|cccc}
\rowcolor{lightgray}{\bf\emph{k}} & \bf Queen & \bf Grid & \bf Hex & \bf Torus & & & {\bf\emph{k}} & \bf GP$(k,1)$ & \bf GP$(k,2)$\\
                         {\bf 2} & 1 & 2 & 2 & 2 & & &     {\bf 5}  & 4  & 4 \\
\rowcolor{lightergray}   {\bf 3} & 2 & 4 & 3 & 3 & & &     {\bf 10} & 8  & 8 \\
                         {\bf 4} & 3 & 7 & 5 & 6 & & &     {\bf 15} & 12 & 11 \\
\rowcolor{lightergray}   {\bf 5} & 4 & 9 & 7 & 9 & & &     {\bf 20} & 16 & 15 \\
                         {\bf 6} & 5 & 13 & 10 & 12 & & &  {\bf 25} & 19 & 19 \\
\rowcolor{lightergray}   {\bf 7} & 6 & 18 & 13 & 17 & & &  {\bf 30} & 23 & 22 \\
                         {\bf 8} & 7 & 23 & 17 & 22 & & &  {\bf 35} & 27 & 25  \\
\rowcolor{lightergray}   {\bf 9} & 7 & 28 & 21 & 27 & & &  {\bf 40} & 30 & 29 \\
                         {\bf 10} & 8 & 35 & 26 & - & & &  {\bf 45} & 34 & 33 \\
\rowcolor{lightergray}   {\bf 11} & - & 42 & 30 & - & & &  {\bf 50} & 38 & 36 \\
                         {\bf 12} & - & - & 36 & - &  & &  {\bf 55} & 42 & 40 \\
\rowcolor{lightergray}   {\bf 13} & - & - & - & - & &  &   {\bf 60} & 46 & - \\
                         {\bf 14} & - & - & - & - & & &    {\bf 65} & 49 & - \\
\rowcolor{lightergray}   {\bf 15} & - & - & - & - & &  &   {\bf 70} & 53 & - \\
                         {\bf 16} & - & - & - & - & &  &   {\bf 75} & 57 & - \\
\rowcolor{lightergray}   {\bf 17} & - & - & - & - & &  &   {\bf 80} & 60 & - \\
\rowcolor{lightergray}   {\bf 18} & - & - & - & - & &  &   {\bf 85} & 64 & - \\

\end{tabular}\caption{The secure domination numbers obtained for various instances from the formulation in this paper. In the cases where an instance timed out, the secure domination number was not found and a dash (-) is listed.\label{tab1}}\end{center}\end{table}

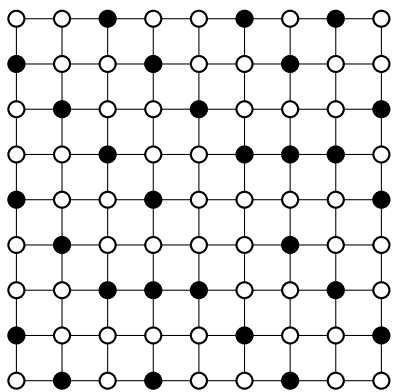
\begin{figure}[h!]\begin{center}
\begin{tikzpicture}[scale=0.6]
\foreach \n in {1,2,3,4,5,6,7,8} \foreach \m in {1,2,3,4,5,6,7,8,9} \draw (\n-1,\m-1) -- (\n,\m-1);
\foreach \n in {1,2,3,4,5,6,7,8,9} \foreach \m in {1,2,3,4,5,6,7,8} \draw (\n-1,\m-1) -- (\n-1,\m);
\foreach \n in {1,2,3,4,5,6,7,8,9} \foreach \m in {1,2,3,4,5,6,7,8,9} \draw[fill=white,thick] (\n-1,\m-1) circle (5pt);
\draw[fill=black] (2,8) circle (5pt); \draw[fill=black] (5,8) circle (5pt); \draw[fill=black] (7,8) circle (5pt); \draw[fill=black] (0,7) circle (5pt);
\draw[fill=black] (3,7) circle (5pt); \draw[fill=black] (6,7) circle (5pt); \draw[fill=black] (1,6) circle (5pt); \draw[fill=black] (4,6) circle (5pt);
\draw[fill=black] (8,6) circle (5pt); \draw[fill=black] (2,5) circle (5pt); \draw[fill=black] (5,5) circle (5pt); \draw[fill=black] (6,5) circle (5pt);
\draw[fill=black] (7,5) circle (5pt); \draw[fill=black] (0,4) circle (5pt); \draw[fill=black] (3,4) circle (5pt); \draw[fill=black] (8,4) circle (5pt);
\draw[fill=black] (1,3) circle (5pt); \draw[fill=black] (6,3) circle (5pt); \draw[fill=black] (2,2) circle (5pt); \draw[fill=black] (3,2) circle (5pt);
\draw[fill=black] (4,2) circle (5pt); \draw[fill=black] (7,2) circle (5pt); \draw[fill=black] (0,1) circle (5pt); \draw[fill=black] (5,1) circle (5pt);
\draw[fill=black] (8,1) circle (5pt); \draw[fill=black] (1,0) circle (5pt); \draw[fill=black] (3,0) circle (5pt); \draw[fill=black] (6,0) circle (5pt);
\end{tikzpicture}\end{center}
\caption{A secure dominating configuration for the $9 \times 9$ grid graph. The solid vertices correspond to those with guards present.\label{fig2}}
\end{figure}

\section{Secure Connected Domination}

We conclude this paper by considering the related problem of {\em secure connected domination}. First, we describe the concept of {\em connected domination}, which was first introduced in 2014 by Sampathkumar and Walikar \cite{sampathkumar}. Simply put, for a graph $G$, a dominating set $D$ is said to be {\em connected dominating} if the subgraph induced by the vertices in $D$ is connected in $G$. Then, for a graph $G$, a dominating set $D$ is said to be {\em secure connected dominating} if the following three requirements are true.

\begin{enumerate}[leftmargin=\parindent,align=left,labelwidth=\parindent,labelsep=0pt]
\item[\bf (R1)] \hspace*{0.12cm} $D$ is secure dominating.
\item[\bf (R2)] \hspace*{0.12cm} $D$ is connected dominating.
\item[\bf (R3)] \hspace*{0.12cm} If an incursion occurs at vertex $j \not\in D$, a designated guard is able to move to defend the incursion, such that the resulting configuration is also connected dominating.
\end{enumerate}

In the short time since the concept of secure connected domination was introduced, a small number results have been obtained. In particular, some bounds and exact values of the secure connected domination numbers have been found for graphs resulting from joins and compositions of graphs were given by Cabaro and Canoy Jr. \cite{cabaro}, while Lad, Reddy and Kumar \cite{lad} proved that the secure connected domination problem is NP-complete in the case of split graphs. Here, we add to the literature on this problem by providing a binary programming formulation for the secure connected domination problem. To the best of our knowledge, this is the first existing formulation of the problem in literature.

First, we paraphrase a formulation for the connected domination problem from \cite{fan}, in which several possible formulations were compared. The standout formulation was a modification of the well-known subtour elimination model by Miller, Tucker and Zemlin \cite{mtz}. First, for the given graph $G$ with vertex set $V$ and edge set $E$, we define a corresponding directed graph $G'$ in the following way. Begin with all the existing vertices of $G$. Each existing edge is replaced by two corresponding directed edges. In keeping with \cite{mtz}, we refer to these collectively as $E \cup E'$. Then, add two new vertices $n+1$ and $n+2$. Finally, add the directed edges $(n+1,i)$, $(n+2,i)$ for $i = 1, \hdots, n$, and the directed edge $(n+1,n+2)$. We refer to the directed edge set of $G'$ as $A$.

Note that in this formulation, two new sets of variables are added, namely binary variables $w_{ij}$ for $i,j = 1, \hdots, n+2$, and continuous variables $u_i \in [1,n+1]$ for $i = V \cup \{n+2\}$. Note that $u_{n+1} = 0$.

\begin{eqnarray}
\sum_{j \in N[i]} x_j & \geq & 1, \qquad\qquad \forall i = 1, \hdots, n,\label{fan1}\\
\sum_{i \in V} w_{n+2,i} & = & 1\label{fan2}\\
\sum_{i : (i,j) \in A} w_{ij} & = & 1, \qquad\qquad \forall j = 1, \hdots, n,\label{fan3}\\
w_{n+1,i} + w_{ij} & \leq & 1, \qquad\qquad \forall (i,j) \in E \cup E'\label{fan4}\\
(n+1)w_{ij} + u_i - u_j + (n-1) w_{ji} & \leq & n, \qquad\qquad \forall (i,j) \in E \cup E'\label{fan5}\\
(n+1)w_{ij} + u_i - u_j & \leq & n, \qquad\qquad \forall (i,j) \in A \setminus (E \cup E')\label{fan6}\\
w_{n+1,n+2} & = & 1,\label{fan7}\\
x_i + w_{n+1,i} & = & 1, \qquad\qquad \forall i = 1, \hdots, n.\label{fan8}
\end{eqnarray}

Note that constraints (\ref{fan1}) are simply the domination constraints. Then, the only link between the original $x_i$ variables and the new variables occurs in constraints (\ref{fan8}).

Now, recall the formulation presented in this paper for the secure domination problem. We define a new binary variable $z_{ik}$ as follows.

$$z_{ik} = \left\{\begin{array}{lll}x_i - y_{ik} & \mbox{ if } & k \in N(i),\\
x_i & \mbox{ if } & k \not\in N(i)\mbox{ and } i \neq k,\\
1 & \mbox{ if } & i = k.\end{array}\right.$$

The interpretation for $z_{ik}$ is as follows. Suppose that an incursion occurs at vertex $k$. If a guard is already present at vertex $k$, it handles the incursion. If not, a guard from an adjacent vertex moves to defend the vertex, and so a new configuration is obtained. Then, $z_{ik} = 1$ if there is a guard present at vertex $i$ in this new configuration, and $z_{ik} = 0$ otherwise. Note that by constraining the $z_{ik}$ variables to be binary, the $y_{ik}$ variables are also forced to be binary.

Then, for each $k = 1, \hdots, n$, we need $\{z_{ik}\}_{i=1}^n$ to be a connected dominating set. This can be easily achieved by applying the constraints (\ref{fan2})--(\ref{fan8}) for each of these, replacing $x_i$ with $z_{ik}$. First, introduce new variables $w^k_{ij}$ and $u^k_i$ for $k = 1, \hdots, n$ to be analogous to the previous $w$ and $u$ variables. Then the formulation is as follows.

\begin{eqnarray}
\sum_{i \in V} w^k_{n+2,i} & = & 1, \qquad\qquad \forall k = 1, \hdots, n,\label{nfan2}\\
\sum_{i : (i,j) \in A} w^k_{ij} & = & 1, \qquad\qquad \forall k,j = 1, \hdots, n,\label{nfan3}\\
w^k_{n+1,i} + w^k_{ij} & \leq & 1, \qquad\qquad \forall k = 1, \hdots, n, \quad (i,j) \in E \cup E'\label{nfan4}\\
(n+1)w^k_{ij} + u^k_i - u^k_j + (n-1) w^k_{ji} & \leq & n, \qquad\qquad \forall k = 1, \hdots, n, \quad \forall (i,j) \in E \cup E'\label{nfan5}\\
(n+1)w^k_{ij} + u^k_i - u^k_j & \leq & n, \qquad\qquad \forall k = 1, \hdots, n, \quad \forall (i,j) \in A \setminus (E \cup E')\label{nfan6}\\
w^k_{n+1,n+2} & = & 1, \qquad\qquad \forall k = 1, \hdots, n,\label{nfan7}\\
z_{ik} + w^k_{n+1,i} & = & 1, \qquad\qquad \forall k,i = 1, \hdots, n.\label{nfan8}
\end{eqnarray}

Finally, we can gather together the sets of constraints which satisfy the three requirements of secure connected domination. Specifically, requirement R1 is satisfied by constraints (\ref{neweq1})--(\ref{neweq6}), requirement R2 is satisfied by constraints (\ref{fan2})--(\ref{fan8}), and requirement R3 is satisfied by constraints (\ref{nfan2})--(\ref{nfan8}). Then, to complete the formulation, we simply add the objective function $\min_{i \in V} x_i$.

We conclude this section by determining the number of variables and constraints in this formulation. First, we note that the $z$ variables are just a notational convenience here, and due to their definition they can be entirely replaced by $x$ and $y$ variables. From constraints (\ref{nfan8}) they will be automatically binary and so we do not need to impose this separately. It can easily be seen that there are $2m + 2n + 1$ variables $w_{ij}$, and hence $2mn + 2n^2 + n$ variables $w^k_{ij}$. Finally, there are $n+1$ variables $u_i$ (recalling that $u_{n+1}$ is defined to be 0), and hence $n^2 + n$ variables $u^k_{ij}$. Hence, there are $O(nm + n^2)$ variables, and it can be easily checked that there is also $O(nm + n^2)$ constraints.

\bibliographystyle{plain}

\end{document}